\documentclass[a4paper, 11pt]{article}
\title{On a bound of $p$-ranks of Iwasawa modules of $\Z_p$-extensions over a quartic CM-field}
\author{Takuya Yanagisawa\thanks{Department of Mathematics, Graduate School of Science, Tokyo University of Science}}
\date{}

\usepackage{amsmath, amssymb}
\usepackage{amsthm}
\usepackage{mathtools}
\usepackage{stmaryrd}
\usepackage{float}

\usepackage[labelsep=none]{caption}

\usepackage[top=3cm,bottom=3cm,left=2cm,right=2cm]{geometry}

\theoremstyle{plain}
\newtheorem{df}{Definition}[section]
\newtheorem{thm}[df]{Theorem}
\newtheorem{prop}[df]{Proposition}
\newtheorem{cor}[df]{Corollary}
\newtheorem{lem}[df]{Lemma}

\theoremstyle{definition}
\newtheorem{rem}[df]{Remark}

\usepackage{tikz}
\usetikzlibrary{arrows.meta}

\newcommand{\Z}{\mathbb{Z}}
\newcommand{\Q}{\mathbb{Q}}

\newcommand{\F}{\mathbb{F}}

\newcommand{\pe}{\mathfrak{p}}
\newcommand{\Pe}{\mathfrak{P}}

\newcommand{\ve}{\varepsilon}

\DeclareMathOperator{\Ker}{Ker}
\DeclareMathOperator{\Gal}{Gal}
\DeclareMathOperator{\Ann}{Ann}
\DeclareMathOperator{\rank}{rank}
\DeclareMathOperator{\Homo}{H}

\newcommand{\ab}{\mathrm{ab}}
\newcommand{\cy}{\mathrm{c}}
\newcommand{\an}{\mathrm{a}}

\newcommand{\tgen}[1]{\overline{\langle #1 \rangle}}
\newcommand{\br}[1]{\lbrack #1 \rbrack}
\newcommand{\bbr}[1]{\llbracket #1 \rrbracket}

\usepackage[backend=biber]{biblatex}
\DeclareFieldFormat{postnote}{#1}
\begin{document}
\maketitle
\begin{abstract}
Let $p$ be a prime number. If a number field $k$ has at least one complex place, there are infinitely many $\Z_p$-extensions over $k$, and some authors studied the behavior of Iwasawa invariants of these $\Z_p$-extensions.  In particular, Fujii studied the case where $k$ is an imaginary quadratic field and obtained some results on the boundedness of Iwasawa $\lambda$-invariants in a certain infinite family of $\Z_p$-extensions. In the present article, we give analogous theorems in the case where $k$ is a quartic CM-field. One of our main theorems determines all the Iwasawa invariants, including the $\nu$-invariants, of a certain infinite family of $\Z_p$-extensions over a quartic CM-field. 
\end{abstract}
\section{Introduction}\label{Section Introduction}
Let $k/\Q$ be a finite extension and $p$ a prime number. If $k$ has at least one complex place, there are infinitely many $\Z_p$-extensions over $k$. In this case, we are interested in the behavior of Iwasawa modules $X(k_\infty)$ as $k_\infty/k$ varies all $\Z_p$-extensions over $k$, where $X(k_\infty)$ is defined to be the Galois group of the maximal unramified abelian pro-$p$-extension $L(k_\infty)/k_\infty$. In particular, we are interested in the behavior of the Iwasawa invariants of these $\Z_p$-extensions. 

The Iwasawa $\lambda$- and $\mu$-invariants of $k_\infty/k$ are defined as structural invariants of the Iwasawa module $X(k_\infty)$ and are widely studied in classical Iwasawa theory, which we will recall in \S\ref{Subsection Fundamental notions and properties of Iwasawa modules}. The behavior of these invariants has been studied by several authors, in particular, in the case where $k$ is an imaginary quadratic field. For example, Sands \cite{Sands1991} studied the case where the $\lambda$-invariant of the cyclotomic $\Z_p$-extension $k^\cy_\infty/k$ is small, and Ozaki \cite{Ozaki2001} studied the relation with a conjecture known as Greenberg's Generalized Conjecture. 

In \cite{Fujii2013}, Fujii also studied the case where $k$ is an imaginary quadratic field and obtained the following. 
\begin{thm}[{\cite[Theorem~1]{Fujii2013}}]\label{FujiiThm1}
Let $k$ be an imaginary quadratic field and $p$ an odd prime number that splits in $k$. Suppose that $\lambda(k^\cy_\infty/k)=2$, where $k^\cy_\infty/k$ is the cyclotomic $\Z_p$-extension. Then $\lambda(k_\infty/k)\leq 2$ for each $\Z_p$-extension $k_\infty/k$ that satisfies $k_\infty\cap k^\an_\infty=k$ and is ramified at each $p$-adic prime, where $k^\an_\infty/k$ is the anti-cyclotomic $\Z_p$-extension (see \S\ref{Subsection Anti-cyclotomic extensions} for the definition). 
\end{thm}
We remark that Fujii also dealt with the case where $p$ does not split in $k$ and obtained similar results. See \cite[Theorem~1 and Proposition~5.1]{Fujii2013} for details. 

In the present article, we improve Fujii's method and study the behavior of Iwasawa modules over a quartic CM-field, instead of an imaginary quadratic field. 

Let $k$ be a CM-field and $p$ an odd prime number that splits completely in $k$. For a set $\Sigma$ of $p$-adic primes of $k$, we define $\overline{\Sigma}\coloneqq\{\overline{\pe}\mid \pe\in\Sigma\}$, where $\overline{\pe}$ is the complex conjugate of $\pe$. For such a set $\Sigma$ and a $\Z_p$-extension $k_\infty/k$, we let $X_\Sigma(k_\infty)$ denote the $\Sigma$-ramified Iwasawa module of $k_\infty/k$, that is, the Galois group of the maximal abelian pro-$p$-extension $M_\Sigma(k_\infty)/k_\infty$ that is unramified outside $\Sigma$. 

The following is one of the main results. 
\begin{thm}\label{minimum rank 2}
Let $k$ be a quartic CM-field, $p$ an odd prime number that splits completely in $k$ and $\Sigma$ a set of $p$-adic primes of $k$ that satisfies $\Sigma\cap\overline{\Sigma}=\emptyset$. Suppose that $X(k^\cy_\infty)\simeq\Z_p^2$ as a $\Z_p\bbr{\Gal(k^\cy_\infty/k)}$-module. Then we have
\[
\min\{\rank_pX_\Sigma(k_\infty), \rank_pX_{\overline{\Sigma}}(k_\infty)\}\leq 2
\]
for each $\Z_p$-extension $k_\infty/k$ that satisfies $k_\infty\cap\widetilde{k}^\an=k$, where $\widetilde{k}^\an/k$ is the maximal multiple anti-cyclotomic $\Z_p$-extension, which we will introduce in \S\ref{Subsection Anti-cyclotomic extensions}. 
\end{thm}
As a corollary to this theorem, we prove the following theorem. 
\begin{thm}\label{infinite Zp2}
Let $k$ be a quartic CM-field and $p$ an odd prime number that splits completely in $k$. Suppose that $X(k^\cy_\infty)\simeq\Z_p^2$ as a $\Z_p\bbr{\Gal(k^\cy_\infty/k)}$-module. Let $k_\infty/k$ be a $\Z_p$-extension that satisfies $k_\infty\cap\widetilde{k}^\an=k$ and is ramified at each $p$-adic prime. Then $X(k_\infty)\simeq\Z_p^2$ as a $\Z_p\bbr{\Gal(k_\infty/k)}$-module and, in particular, $\mu(k_\infty/k)=0$ and $\lambda(k_\infty/k)=2$. Moreover, if we let $e\geq 0$ be the minimum integer such that $k_\infty/k_e$ is totally ramified at each $p$-adic prime, then $v_p(h_{k_n})=2n+v_p(h_{k_e})-2e$ holds for each $n\geq e$ and, in particular, $\nu(k_\infty/k)=v_p(h_{k_e})-2e$. 
\end{thm}
In \S\ref{Subsection Examples}, we also give a criterion to find pairs $(k, p)$ that satisfy the assumptions in the above theorems (see Lemma~\ref{condition}). We give numerical examples by using this criterion. 

We also give another result on CM-fields of arbitrary degree (Theorem~\ref{infinite Zpd}), but we have no numerical applications yet, except for imaginary quadratic fields. 

The structure of this paper is as follows. In \S\ref{Section Preliminaries}, we introduce some notions and properties of Iwasawa modules that will frequently be used throughout this paper. In \S\ref{Section Strategy}, we present several key lemmas to the proof of the main theorems. In \S\ref{Section Iwasawa modules over a quartic CM-field}, we deal with quartic CM-fields and prove the main theorems stated above. In \S\ref{Section Iwasawa modules over a CM-field}, we show some results that deal with CM-fields of arbitrary degree. 
\section{Preliminaries}\label{Section Preliminaries}
\subsection{Fundamental notion and properties of Iwasawa modules}\label{Subsection Fundamental notions and properties of Iwasawa modules}
Fix an algebraic closure $\overline{\Q}$ of the field $\Q$ of rational numbers and we consider all algebraic extensions of $\Q$ to be in $\overline{\Q}$. 

Fix a prime number $p$. For any algebraic extension $k/\Q$, let $L(k)/k$ be the maximal unramified abelian pro-$p$-extension and $X(k)$ the Galois group $\Gal(L(k)/k)$. Let $K/k$ be a $\Z_p^d$-extension for some $d\geq 1$. Then $L(K)/k$ is a Galois extension by the maximality of $L(K)$, and $\Gal(K/k)$ acts on $X(K)$ as inner automorphism. Moreover, $X(K)$ naturally has a module structure over the completed group algebra $\Z_p\bbr{\Gal(K/k)}$. Once we fix a $\Z_p$-basis $\tau_1, \ldots, \tau_d$ of $\Gal(K/k)$, there is an isomorphism $\Z_p\bbr{\Gal(K/k)}\simeq\Z_p\bbr{T_1, \ldots, T_d}$ defined by $\tau_1\leftrightarrow 1+T_1, \ldots, \tau_d\leftrightarrow 1+T_d$, where $\Z_p\bbr{T_1, \ldots, T_d}$ is the ring of formal power series in $d$ variables. 

Let $k/\Q$ be a finite extension and $k_\infty/k$ a $\Z_p$-extension. For each non-negative integer $n$, there is a unique intermediate field $k_n$ of $k_\infty/k$ such that $[k_n:k]=p^n$. In this situation, Iwasawa's class number formula states that there are non-negative integers $\lambda(k_\infty/k)$, $\mu(k_\infty/k)$ and an integer $\nu(k_\infty/k)$ such that the $p$-order $v_p(h_{k_n})$ of the class number $h_{k_n}$ of $k_n$ satisfies
\[
v_p(h_{k_n})=\lambda(k_\infty/k)n+\mu(k_\infty/k)p^n+\nu(k_\infty/k)
\]
for sufficiently large $n\geq 0$. These invariants $\lambda(k_\infty/k)$, $\mu(k_\infty/k)$ and $\nu(k_\infty/k)$ are called the Iwasawa $\lambda$-, $\mu$- and $\nu$- invariants of $k_\infty/k$, respectively. Although the $\lambda$- and $\mu$-invariants of $k_\infty/k$ are classically defined as structural invariants of $X(k_\infty)$ as a $\Z_p\bbr{\Gal(k_\infty/k)}$-module, we only use the following properties of these invariants. 
\begin{prop}[{\cite[p.293, Remark~3 and p.300, Exercise~2~(c)]{NSW2008}}]\label{mu and lambda}
Keep the settings as above. Then the following hold. 
\begin{enumerate}
\item $\mu(k_\infty/k)=0$ if and only if $X(k_\infty)$ is finitely generated as a $\Z_p$-module. 

\item $\lambda(k_\infty/k)=\rank_{\Z_p}X(k_\infty)$. 
\end{enumerate}
\end{prop}
Here we use the notation $\rank_{\Z_p}M\coloneqq\dim_{\Q_p}(\Q_p\otimes_{\Z_p}M)$ for any $\Z_p$-module $M$. Proposition~\ref{mu and lambda} follows directly from the classical definition of the $\mu$- and $\lambda$-invariants. 

For a finite extension $k/\Q$ and a prime number $p$, there is at least one $\Z_p$-extension, namely the cyclotomic $\Z_p$-extension, which we denote by $k^\cy_\infty/k$. Moreover, the following theorem describes the number of independent $\Z_p$-extensions over $k$. 
\begin{thm}[{\cite[Theorem~13.4]{Washington1997}}]
Let $k/\Q$ be a finite extension and $p$ a prime number. Let $\widetilde{k}$ be the maximal multiple $\Z_p$-extension of $k$, that is, the composite of all $\Z_p$-extensions of $k$. Then $\Gal(\widetilde{k}/k)\simeq\Z_p^{r_2+1+\delta}$, where $r_2$ is the number of complex places of $k$ and $\delta$ is the defect of Leopoldt's Conjecture for $k$ and $p$. 
\end{thm}
\begin{rem}
Leopoldt's Conjecture is equivalent to $\delta=0$ and is known to hold for abelian number fields. 
\end{rem}
This theorem implies that if there is at least one complex place of $k$, then there are infinitely many $\Z_p$-extensions over $k$. 

Now we introduce some slightly generalized notion of Iwasawa modules. Let $k/\Q$ be an algebraic extension and $p$ a prime number. We let $S_p(k)$ denote the set of all $p$-adic primes of $k$. For a subset $\Sigma$ of $S_p(k)$, we denote $M_\Sigma(k)/k$ the maximal abelian pro-$p$-extension that is unramified at all primes not contained in $\Sigma$, and $X_\Sigma(k)$ its Galois group. Let $K/k$ be an algebraic extension and put $\Sigma_K\coloneqq\{\Pe\in S_p(K)\mid\Pe\cap k\in\Sigma\}$. We use the notation $M_\Sigma(K)\coloneqq M_{\Sigma_K}(K)$ and $X_\Sigma(K)\coloneqq X_{\Sigma_K}(K)$ for simplicity. We also use the notation $M_\pe(K)\coloneqq M_{\{\pe\}}(K)$ and $X_\pe(K)\coloneqq X_{\{\pe\}}(K)$ for $\pe\in S_p(k)$. We remark that $M_\emptyset(K)=L(K)$ and $X_\emptyset(K)=X(K)$. When $K/k$ is a $\Z_p^d$-extension for some $d\geq 1$, we regard $X_\Sigma(K)$ as a $\Z_p\bbr{\Gal(K/k)}$-module as in the case where $\Sigma=\emptyset$, which we mentioned above, and call $X_\Sigma(K)$ the $\Sigma$-ramified Iwasawa module of $K/k$. 

We show some lemmas that we will frequently use in the present article. 
\begin{lem}\label{ur}
Let $k/\Q$ be a finite extension and $p$ a prime number that splits completely in $k$. Let $k_\infty/k$ be a $\Z_p$-extension that is ramified at each $p$-adic prime. Then $\widetilde{k}/k_\infty$ is unramified. Moreover, if $p$ is odd, then $M_p(k)/k_\infty$ is unramified. 
\end{lem}
\begin{proof}
For example, see \cite[Proposition~1]{Ozaki1997}. 
\end{proof}
For a profinite group $G$ and a profinite $G$-module $M$, we denote the maximal $G$-coinvariant quotient of $M$ by $M_G$. For example, if $G$ is isomorphic to $\Z_p$ and $M$ is a pro-$p$-group, then $M$ is a $\Z_p\bbr{G}$-module and $M_G=M/(\sigma-1)M$, where $\sigma$ is any topological generator of $G$. 
\begin{lem}\label{fiveterm}
Let $p$ be a prime number, $k/\Q$ an algebraic extension, $F/k$ a pro-$p$-extension and $K/F$ an abelian pro-$p$-extension such that $K/k$ is a Galois extension. Let $\Sigma$ be a set of $p$-adic primes of $k$. Then there is an exact sequence
\[
\begin{tikzpicture}
\node (a) at (-2,0) {$\bigwedge^2\Gal(K/F)$}; 
\node (b) at (2,0) {$X_{\Sigma}(K)_{\Gal(K/F)}$}; 
\node (c) at (6.5,0) {$\Gal(M_\Sigma(K)\cap F^\ab/F)$}; 
\node (d) at (10.3,0) {$\Gal(K/F)$}; 
\node (e) at (12.4,0) {$0$};
\draw[->] (a) -- (b); \draw[->] (b) -- (c); \draw[->] (c) -- (d); \draw[->] (d) -- (e); 
\end{tikzpicture}
\]
of $\Z_p\bbr{\Gal(F/k)}$-modules. Here, $\bigwedge^2$ is the exterior product over $\Z_p$ and $F^\ab/F$ is the maximal abelian extension. 
\end{lem}
\begin{proof}
Put $\Gamma\coloneqq\Gal(K/F)$. Then there is an exact sequence
\[
\begin{tikzpicture}
\node (a) at (0,0) {$\Homo_2(\Gamma, \Z_p)$}; 
\node (b) at (3.5,0) {$\Homo_1(X_\Sigma(K), \Z_p)_\Gamma$}; 
\node (c) at (8,0) {$\Homo_1(\Gal(M_\Sigma(K)/F), \Z_p)$}; 
\node (d) at (12.2,0) {$\Homo_1(\Gamma, \Z_p)$}; 
\node (e) at (14.6,0) {$0$};
\draw[->] (a) -- (b); \draw[->] (b) -- (c); \draw[->] (c) -- (d); \draw[->] (d) -- (e); 
\end{tikzpicture}
\]
of $\Z_p\bbr{\Gal(F/k)}$-modules, which is a part of the five term exact sequence in group homology (see \cite[Corollary~7.2.6]{RZ2010}, for example). One can see $\Homo_1(\Gamma, \Z_p)\simeq\Gamma^\ab=\Gamma$ and $\Homo_2(\Gamma, \Z_p)\simeq\bigwedge^2\Gamma$, by general group theory (see \cite[11.4.4, 11.4.16]{Robinson1996} and \cite[Lemma~6.8.6~(b)]{RZ2010}, as well). We also have $\Homo_1(\Gal(M_\Sigma(K)/F), \Z_p)\simeq\Gal(M_\Sigma(K)/F)^\ab=\Gal(M_\Sigma(K)\cap F^\ab/F)$. We complete the proof by combining the above. 
\end{proof}
For a given $\Z_p$-module $M$, we use the notation $\rank_p M\coloneqq\dim_{\F_p}(\F_p\otimes_{\Z_p}M)$, the $p$-rank of $M$. We use the following lemma, which is a direct consequence of Lemma~\ref{fiveterm}, to bound $p$-ranks of Iwasawa modules of $\Z_p$-extensions. 
\begin{lem}\label{prankineq}
In the setting of Lemma~\ref{fiveterm}, we have 
\[
\rank_p X_\Sigma(F)\leq \rank_p X_\Sigma(K)_{\Gal(K/F)}+\rank_p\Gal(K/F).
\] 
\end{lem}
We also use the following well-known lemma. 
\begin{lem}\label{central}
Let $k/\Q$ be an algebraic extension, $p$ a prime number and $K/k$ a $\Z_p$-extension. Let $\Sigma$ be a set of $p$-adic primes of $k$. Then we have $X_\Sigma(K)_{\Gal(K/k)}=\Gal(M_\Sigma(K)\cap k^\ab/K)$.
\end{lem}
\begin{proof}
We have an exact sequence
\[
\begin{tikzpicture}
\node (a) at (-3,0) {$0$}; 
\node (b) at (0,0) {$X_\Sigma(K)_{\Gal(K/k)}$}; 
\node (c) at (4,0) {$\Gal(M_\Sigma(K)\cap k^\ab/k)$}; 
\node (d) at (8,0) {$\Gal(K/k)$}; 
\node (e) at (10.5,0) {$0$};
\draw[->] (a) -- (b); \draw[->] (b) -- (c); \draw[->] (c) -- (d); \draw[->] (d) -- (e); 
\end{tikzpicture}
\]
of $\Z_p$-modules by Lemma~\ref{fiveterm}, because $\bigwedge^2\Gal(K/k)=0$ since $\Gal(K/k)$ is cyclic as a $\Z_p$-module. Then it follows that
\[
X_\Sigma(K)_{\Gal(K/k)}=\Ker(\Gal(M_\Sigma(K)\cap k^\ab/k)\to\Gal(K/k))=\Gal(M_\Sigma(K)\cap k^\ab/K), 
\]
as desired. 
\end{proof}
\subsection{Anti-cyclotomic extensions}\label{Subsection Anti-cyclotomic extensions}
Let $k$ be a CM-field and $p$ an odd prime number. We denote the maximal totally real subfield of $k$ by $k^+$ and let $J\in\Gal(k/k^+)$ be the complex conjugation. Put $\ve_+\coloneqq(1+J)/2, \ve_-\coloneqq(1-J)/2\in\Z_p\br{\Gal(k/k^+)}$, and for a given $\Z_p\br{\Gal(k/k^+)}$-module $M$, we define $M^+\coloneqq \ve_+M$ and $M^-\coloneqq \ve_-M$. Then we obtain a decomposition $M=M^+\oplus M^-$. 

Put $\Gamma\coloneqq\Gal(\widetilde{k}/k)$. Then we have a decomposition $\Gamma=\Gamma^+\oplus\Gamma^-$ as mentioned above. We define $\widetilde{k}^\an\coloneqq(\widetilde{k})^{\Gamma^+}$ and call it the maximal multiple anti-cyclotomic $\Z_p$-extension of $k$. Also, a $\Z_p$-extension $k_\infty/k$ is said to be an anti-cyclotomic $\Z_p$-extension if $k_\infty\subseteq\widetilde{k}^\an$. In other words, a given $\Z_p$-extension $k_\infty/k$ is an anti-cyclotomic $\Z_p$-extension if and only if $k_\infty/k^+$ is a Galois extension such that $J$ acts on $\Gal(k_\infty/k)$ as inverse. 

We remark that $(\widetilde{k})^{\Gamma^-}$ coincides with $k^\cy_\infty$ if Leopoldt's Conjecture holds for $k$ and $p$. 
\begin{lem}\label{Galois over real}
Let $k$ be a CM-field and $p$ an odd prime number. Suppose that Leopoldt's Conjecture holds for $k$ and $p$. Then any intermediate field $K$ of $\widetilde{k}/k^\cy_\infty$ is a Galois extension over $k^+$. 
\end{lem}
\begin{proof}
While $\Gal(\widetilde{k}/K)$ is a closed subgroup of $\Gal(\widetilde{k}/k^\cy_\infty)$, it is also a submodule as a $\Gal(k/k^+)$-module since $J$ acts on $\Gal(\widetilde{k}/k^\cy_\infty)$ as inverse. Thus $\Gal(k/k^+)$ acts on the quotient $\Gal(K/k)=\Gal(\widetilde{k}/k)/\Gal(\widetilde{k}/K)$ as well. This means that $K/k^+$ is a Galois extension. 
\end{proof}
If $k$ is an imaginary quadratic field, then the maximal multiple anti-cyclotomic $\Z_p$-extension $\widetilde{k}^\an/k$ is a $\Z_p$-extension, so we write $k^\an_\infty\coloneqq\widetilde{k}^\an$. Then any $\Z_p$-extension $k_\infty/k$ clearly satisfies $k_\infty\subseteq k^\cy_\infty k^\an_\infty$, since $k^\cy_\infty k^\an_\infty=\widetilde{k}$. The following Lemma \ref{anti-cyclotomic intersection} allows us to deal with $\Z_p$-extensions over a CM-field like those over an imaginary quadratic field. We will use Lemma~\ref{anti-cyclotomic intersection} in the proof of Theorems \ref{infinite Zp2}  and \ref{infinite Zpd}. 
\begin{lem}\label{anti-cyclotomic intersection}
Let $k$ be a CM-field and $p$ an odd prime number. Suppose that Leopoldt's Conjecture holds for $k$ and $p$. Then, for each $\Z_p$-extension $k_\infty/k$ such that $k_\infty\neq k^\cy_\infty$, there exists a unique anti-cyclotomic $\Z_p$-extension $k^\an_\infty/k$ such that $k_\infty\subseteq k^\cy_\infty k^\an_\infty$. In addition, this anti-cyclotomic $\Z_p$-extension satisfies $k_\infty\cap k^\an_\infty=k_\infty\cap\widetilde{k}^\an$. 
\end{lem}
\begin{proof}
Put $K\coloneqq k^\cy_\infty k_\infty$. Then $K/k$ is a $\Z_p^2$-extension that contains $k^\cy_\infty$. By Lemma~\ref{Galois over real}, the extension $K/k^+$ is a Galois extension, so $\Gal(K/k)$ is a $\Z_p\br{\Gal(k/k^+)}$-module. Now, $k^\an_\infty\coloneqq K^{\Gal(K/k)^+}$ is the desired anti-cyclotomic $\Z_p$-extension of $k$. We can also see $k_\infty\cap k^\an_\infty=k_\infty\cap\widetilde{k}^\an$ by the construction of $k^\an_\infty$. The uniqueness of $k^\an_\infty$ also follows from the construction. 
\end{proof}
\subsection{Lemmas on $\Sigma$-ramified Iwasawa modules over CM-fields}\label{Subsection Lemmas on Sigma-ramified Iwasawa modules over CM-fields}
Let $k$ be a CM-field and $p$ an odd prime number. For a given set $\Sigma$ of $p$-adic primes of $k$, we define $\overline{\Sigma}\coloneqq\{\overline{\pe}\mid \pe\in\Sigma\}$, where $\overline{\pe}$ is the complex conjugate of $\pe$. 
\begin{lem}\label{Xp0}
Let $F$ be a totally real number field and $p$ an odd prime number that splits completely in $F$. Suppose that $X(F^\cy_\infty)=0$. Then $X_p(F^\cy_\infty)=0$. 
\end{lem}
\begin{proof}
Since $M_p(F)=F^\cy_\infty$ by Lemma~\ref{ur}, we have $X_p(F^\cy_\infty)_{\Gal(F^\cy_\infty/F)}=\Gal(M_p(F^\cy_\infty)\cap F^\ab/F^\cy_\infty)=\Gal(M_p(F)/F^\cy_\infty)=0$ by Lemma~\ref{central}. Thus we have $X_p(F^\cy_\infty)=0$ by Nakayama's Lemma. 
\end{proof}
The following lemma is a generalization of \cite[Lemma~2]{Fujii2014}, which is proved in the same way. 
\begin{lem}\label{M_Sigma L}
Let $k$ be a CM-field, $p$ an odd prime number that splits completely in $k$ and $\Sigma$ a set of $p$-adic primes of $k$ that satisfies $\Sigma\cap\overline{\Sigma}=\emptyset$. Suppose that $X(k^\cy_\infty)^+=0$. Then $M_\Sigma(k^\cy_\infty)=L(k^\cy_\infty)$. 
\end{lem}
\begin{proof}
Put $F\coloneqq k^+$ and consider the natural action of $\Gal(k^\cy_\infty/F^\cy_\infty)$ on $X_p(k^\cy_\infty)$. Then, by Lemma~\ref{Xp0}, the generator of $\Gal(k^\cy_\infty/F^\cy_\infty)$ acts on $X_p(k^\cy_\infty)$ as inverse. Thus each intermediate field of $M_p(k^\cy_\infty)/k^\cy_\infty$ is a Galois extension of $F^\cy_\infty$. In particular, $M_\Sigma(k^\cy_\infty)/F^\cy_\infty$ is a Galois extension. Since each $p$-adic prime of $F^\cy_\infty$ splits in $k^\cy_\infty$, the extension $M_\Sigma(k^\cy_\infty)/F^\cy_\infty$ is unramified outside $\Sigma$. Considering the action of $\Gal(k^\cy_\infty/F^\cy_\infty)$, we conclude that $M_\Sigma(k^\cy_\infty)/F^\cy_\infty$ is unramified outside $\overline{\Sigma}$, hence it is unramified everywhere by the assumption that $\Sigma\cap\overline{\Sigma}=\emptyset$.  This means $M_\Sigma(k^\cy_\infty)\subseteq L(k^\cy_\infty)$. The reverse inclusion is clear. 
\end{proof}
\section{Strategy}\label{Section Strategy}
In this section, we review the logical structure of Fujii's proof of Theorem~\ref{FujiiThm1} from our perspective. Here we mention that Fujii's method is based on Bloom's method in \cite[\S 5]{Bloom1979}. 
\subsection{A lemma on power series}\label{Subsection Lemmas on power series}
Let $p$ be an odd prime number and $f(T)\in\Z_p\bbr{T}$ a power series. When $f(T)\neq 0$, by the $p$-adic Weierstrass preparation theorem, we can uniquely write $f(T)=p^mg(T)U(T)$ with a non-negative integer $m$, a distinguished polynomial $g(T)\in\Z_p\br{T}$ and a unit power series $U(T)\in\Z_p\bbr{T}^{\times}$. We put $\mu(f(T))\coloneqq m$ and $\lambda(f(T))\coloneqq\deg g(T)$. When $f(T)=0$, we define $\mu(0)\coloneqq\infty$. 

For $\alpha\in\Z_p$, we define an ideal $I_{\alpha}$ of $\Z_p\bbr{S, T}$ by
\[
I_{\alpha}\coloneqq((1+S)^{\alpha}(1+T)-1, S-f(T), p). 
\]
Now our goal is to show Lemma~\ref{powerseries dimension} below. The following result is given by Fujii (see \cite[Lemma~3.3]{Fujii2013}). Here we give a slightly different calculation. 
\begin{lem}\label{powerseries dimension}
For each $\alpha\in\Z_p$, one of the following holds. 
\begin{enumerate}
\item $\dim_{\F_p}\Z_p\bbr{S, T}/I_{\alpha}\leq 1$. 

\item $\dim_{\F_p}\Z_p\bbr{S, T}/I_{-\alpha}\leq 1$. 
\end{enumerate}
\end{lem}
\begin{proof}
Put
\[
h_{\alpha}(T)\coloneqq (1+f(T))^{-\alpha}-(1+T), 
\] 
so that we have 
\[
I_{\alpha}=((1+S)^{-\alpha}-(1+T), S-f(T), p)=(h_{\alpha}(T), S-f(T), p). 
\]
Then we obtain an isomorphism
\[
\Z_p\bbr{S, T}/I_{\alpha}\simeq \Z_p\bbr{T}/(h_{\alpha}(T), p), 
\]
defined by sending $S$ to $f(T)$. So, once we have shown $\mu(h_{\alpha}(T))= 0$, we obtain 
\[
\dim_{\F_p}\Z_p\bbr{S, T}/I_{\alpha}=\lambda(h_{\alpha}(T)), 
\]
thus it suffices to calculate $\lambda(h_{\alpha}(T))$. 

We first deal with the case where $\mu(f(T))>0$. Then we have
\[
h_{\alpha}(T)\equiv1-(1+T)=-T \pmod{p}, 
\]
and thus we have $\lambda(h_{\alpha}(T))=1$. 

In what follows, we assume that $\mu(f(T))=0$. If we further assume that $\lambda(f(T))=0$, we have $f(T)=U(T)\in\Z_p\bbr{T}^\times$, and hence $S-f(T)\in\Z_p\bbr{S, T}^\times$. This implies $\Z_p\bbr{S, T}/I_{\alpha}=0$. Therefore, we may also assume that $\lambda(f(T))\geq 1$. 

Put $u\coloneqq v_p(\alpha)$ and $\alpha'\coloneqq-\alpha p^{-u}\in\Z_p^\times$, in the case where $\alpha\neq 0$. If $\alpha=0$, then we put $u\coloneqq1$ and $\alpha'\coloneqq0$. In any case, we have
\begin{align*}
h_{\alpha}(T)&\equiv(1+f(T)^{p^{u}})^{\alpha'}-(1+T)=\sum_{i=1}^\infty\binom{\alpha'}{i}f(T)^{p^{u}i}-T\\
&\equiv T^{\lambda(f(T))p^{u}}\sum_{i=1}^\infty\binom{\alpha'}{i}T^{\lambda(f(T))p^{u}(i-1)}U(T)^{p^{u}i}-T \pmod{p}. 
\end{align*}

If $u\geq 1$, then
\[
h_{\alpha}(T)T^{-1}\equiv T^{\lambda(f(T))p^{u}-1}\sum_{i=1}^\infty\binom{\alpha'}{i}T^{\lambda(f(T))p^{u}(i-1)}U(T)^{p^{u}i}-1 \pmod{p}, 
\] 
and the right hand side is in $\Z_p\bbr{T}^\times$. Therefore, it holds that $\lambda(h_{\alpha}(T))=1$. 

Now suppose that $u=0$. If $\lambda(f(T))>1$, then
\[
h_{\alpha}(T)T^{-1}\equiv T^{\lambda(f(T))-1}\sum_{i=1}^\infty\binom{\alpha'}{i}T^{\lambda(f(T))(i-1)}U(T)^{i}-1 \pmod{p}, 
\]
and the right hand side is in $\Z_p\bbr{T}^\times$, so it holds that $\lambda(h_{\alpha}(T))=1$. 

If $\lambda(f(T))=1$, then
\[
h_{\alpha}(T)T^{-1}\equiv\sum_{i=1}^\infty\binom{\alpha'}{i}T^{\lambda(f(T))(i-1)}U(T)^{i}-1 \pmod{p}, 
\]
and thus
\[
\left[ h_{\alpha}(T)T^{-1} \right]_{T=0}\equiv\alpha'U(0)-1 \pmod{p}. 
\]
If $\alpha'U(0)-1\not\equiv 0\pmod{p}$ holds, then $h_{\alpha}(T)T^{-1}\in\Z_p\bbr{T}^\times$, so it holds that $\lambda(h_{\alpha}(T))=1$. If, to the contrary, $\alpha'U(0)-1\equiv 0\pmod{p}$ holds, then we have
\[
-\alpha'U(0)-1\equiv-\alpha'U(0)-1+(\alpha'U(0)-1)=-2\not\equiv 0\pmod{p}, 
\]
so $h_{-\alpha}(T)T^{-1}\in\Z_p\bbr{T}^\times$. Therefore, it holds that $\lambda(h_{-\alpha}(T))=1$. This completes the proof of Lemma~\ref{powerseries dimension}. 
\end{proof}
\subsection{Bounding $p$-ranks of certain Galois-coinvariant modules}\label{Subsection Bounding p-ranks of certain Galois-coinvariant modules}
\begin{thm}\label{First Theorem}
Let $k$ be a CM-field and $p$ an odd prime number that splits completely in $k$. Let $k^\an_\infty/k$ be an anti-cyclotomic $\Z_p$-extension. Put $K\coloneqq k^\cy_\infty k^\an_\infty$ and let $X$ be a $\Z_p\bbr{\Gal(K/k)}$-module that is cyclic as a $\Z_p\bbr{\Gal(K/k^\cy_\infty)}$-module. Let $k_\infty/k$ be a $\Z_p$-extension that satisfies $k_\infty\subseteq K$ and $k_\infty\cap k^\an_\infty=k$. Then one of the following holds.  
\begin{enumerate}
\item $\rank_p X_{\Gal(K/k_\infty)}\leq1$. 

\item $\rank_p X_{\Gal(K/k'_\infty)}\leq1$, where $k'_\infty$ is the complex conjugation of $k_\infty$. 
\end{enumerate}
\end{thm}
\begin{proof}
Take topological generators $\sigma$ of $\Gal(K/k^\an_\infty)$ and $\tau$ of $\Gal(K/k^\cy_\infty)$, and fix an isomorphism $\Z_p\bbr{\Gal(K/k)}\simeq\Z_p\bbr{S, T}$ defined by $\sigma\leftrightarrow 1+S$, $\tau\leftrightarrow1+T$. Put $\Lambda\coloneqq\Z_p\bbr{S, T}$ and we consider $X$ as a $\Lambda$-module. Then we have a power series $f(T)\in\Z_p\bbr{T}$ such that $S-f(T)\in\Ann_\Lambda X$, as follows (see \cite[p.249]{Bloom1979} and \cite[Lemma~3.2]{Fujii2013}, as well). Here, $\Ann_\Lambda X\coloneqq\{r\in\Lambda\mid rX=0\}$ is the annihilator ideal of $X$. Since we are assuming that $X$ is cyclic as a $\Z_p\bbr{\Gal(K/k^\cy_\infty)}$-module, we have an element $x\in X$ such that $X=\Z_p\bbr{T}x$. Here we note that the fixed  isomorphism $\Z_p\bbr{\Gal(K/k)}\simeq\Z_p\bbr{S, T}$ induces an isomorphism $\Z_p\bbr{\Gal(K/k^\cy_\infty)}\simeq\Z_p\bbr{T}$. Thus we obtain a power series $f(T)\in\Z_p\bbr{T}$ such that $Sx=f(T)x$. Then $(S-f(T))x=0$, so $S-f(T)\in\Ann_\Lambda X$. Such a power series $f(T)$ is not unique in general, but we fix one here. 

Write $k_\infty=K^{\tgen{\sigma^\alpha\tau}}$ with $\alpha\in\Z_p$. Here we remark that $k'_\infty=K^{\tgen{\sigma^{\alpha}\tau^{-1}}}=K^{\tgen{\sigma^{-\alpha}\tau}}$. By the definition of $f(T)$, we have a surjective morphism $\Lambda/(S-f(T))\to X$. Then we also have a surjective morphism 
\[
\Lambda/((1+S)^{\alpha}(1+T)-1, S-f(T))\to X_{\Gal(K/k_\infty)},
\]
since $X_{\Gal(K/k_\infty)}=X/((1+S)^{\alpha}(1+T)-1)$. Thus we obtain
\[
\rank_pX_{\Gal(K/k_\infty)}\leq \dim_{\F_p}\Z_p\bbr{S, T}/I_{\alpha},  
\]
where $I_{\alpha}$ is defined in Section \ref{Subsection Lemmas on power series}. We also have
\[
\rank_pX_{\Gal(K/k'_\infty)}\leq \dim_{\F_p}\Z_p\bbr{S, T}/I_{-\alpha} 
\]
for the same reason. Therefore, we have the assertion by Lemma~\ref{powerseries dimension}. 
\end{proof}
\begin{rem}\label{isomorphic}
Though this theorem gives us information about one of $k_\infty$ and $k'_\infty$, we will be able to obtain information about both of $k_\infty$ and $k'_\infty$, since $X_\Sigma(k_\infty)$ and $X_{\overline{\Sigma}}(k'_\infty)$ are isomorphic as $\Z_p$-modules for a set $\Sigma$ of $p$-adic primes of $k$. 
\end{rem}
\begin{rem}\label{deduceFujii}
We can deduce Theorem~\ref{FujiiThm1} from Theorem~\ref{First Theorem}. This is a part of Fujii's proof of Theorem~\ref{FujiiThm1} and is not new, but we prefer to briefly recall the method here. Suppose the conditions as in Theorem~\ref{FujiiThm1}. Then there is an exact sequence
\[
\begin{tikzpicture}
\node (a) at (-0.5,0) {$0$}; 
\node (b) at (2,0) {$X(\widetilde{k})_{\Gal(\widetilde{k}/k^\cy_\infty)}$}; 
\node (c) at (5,0) {$X(k^\cy_\infty)$}; 
\node (d) at (8,0) {$\Gal(\widetilde{k}/k^\cy_\infty)$}; 
\node (e) at (10.5,0) {$0$};
\draw[->] (a) -- (b); \draw[->] (b) -- (c); \draw[->] (c) -- (d); \draw[->] (d) -- (e); 
\end{tikzpicture}
\]
by Lemmas \ref{ur} and \ref{fiveterm}. Since $X(k^\cy_\infty)\simeq\Z_p^2$, we have $X(\widetilde{k})_{\Gal(\widetilde{k}/k^\cy_\infty)}\simeq\Z_p$, and thus $X(\widetilde{k})$ is cyclic as a $\Z_p\bbr{\Gal(\widetilde{k}/k^\cy_\infty)}$-module by Nakayama's Lemma. Then by Lemma~\ref{prankineq}, we obtain
\[
\lambda(k_\infty/k)\leq\rank_pX(k_\infty)\leq\rank_pX(\widetilde{k})_{\Gal(\widetilde{k}/k_\infty)}+1. 
\]
Recalling that $X(k_\infty)$ and $X(k'_\infty)$ are isomorphic as $\Z_p$-modules, we can deduce Theorem~\ref{FujiiThm1} from Theorem~\ref{First Theorem} applied to $X=X(\widetilde{k})$. 

Here we remark that the assumption on the ramification in $k_\infty/k$ is not needed in Theorem~\ref{FujiiThm1}. 
\end{rem}
\section{Iwasawa modules over a quartic CM-field}\label{Section Iwasawa modules over a quartic CM-field}
\subsection{Main theorems}\label{Subsection Main theorems}
In this subsection, we prove the main theorems. Before proving Theorem~\ref{minimum rank 2}, we give some remarks. 
\begin{rem}
Leopoldt's Conjecture holds for any quartic CM-field $k$ and any prime number $p$. Indeed, the $p$-adic regulator of $k$ is simply the value of the $p$-adic logarithm at a fundamental unit of $k$, which is non-zero. 
\end{rem}
\begin{rem}\label{trivial action}
Let $k$ be a quartic CM-field, $p$ a prime number that splits completely in $k$ and $k_\infty/k$ a $\Z_p$-extension that is ramified at each $p$-adic prime. Then the following are equivalent. 
\begin{enumerate}
\item $X(k_\infty)\simeq\Z_p^2$ as a $\Z_p\bbr{\Gal(k_\infty/k)}$-module. 

\item $X(k_\infty)\simeq\Z_p^2$ as a $\Z_p$-module. 

\item $L(k_\infty)=\widetilde{k}$. 
\end{enumerate}
This is because there is a surjection $X(k_\infty)\to\Gal(\widetilde{k}/k_\infty)$ of $\Z_p\bbr{\Gal(k_\infty/k)}$-modules by Lemma~\ref{ur}. In particular, if $X(k^\cy_\infty)\simeq\Z_p^2$ (and $p$ is odd), then $X(k^\cy_\infty)^+=0$. 
\end{rem}
\begin{proof}[Proof of Theorem~\ref{minimum rank 2}]
Using Lemma~\ref{anti-cyclotomic intersection}, we can take an anti-cyclotomic $\Z_p$-extension $k^\an_\infty/k$ such that $k_\infty\subseteq k^\cy_\infty k^\an_\infty$, and this satisfies $k_\infty\cap k^\an_\infty=k$. Put $K\coloneqq k^\cy_\infty k^\an_\infty$. Then there is an exact sequence
\[
\begin{tikzpicture}
\node (a) at (-0.5,0) {$0$}; 
\node (b) at (2,0) {$X_\Sigma(K)_{\Gal(K/k^\cy_\infty)}$}; 
\node (c) at (5,0) {$X(k^\cy_\infty)$}; 
\node (d) at (8,0) {$\Gal(K/k^\cy_\infty)$}; 
\node (e) at (10.5,0) {$0$};
\draw[->] (a) -- (b); \draw[->] (b) -- (c); \draw[->] (c) -- (d); \draw[->] (d) -- (e); 
\end{tikzpicture}
\]
by Lemmas \ref{ur}, \ref{fiveterm} and \ref{M_Sigma L}. By the assumption that $X(k^\cy_\infty)\simeq\Z_p^2$, we have $X_\Sigma(K)_{\Gal(K/k^\cy_\infty)}\simeq\Z_p$, and thus $X_\Sigma(K)$ is cyclic as a $\Z_p\bbr{\Gal(K/k^\cy_\infty)}$-module by Nakayama's Lemma. Hence, by Theorem~\ref{First Theorem} applied to $X=X_\Sigma(K)$, we have one of the following two inequalities. 
\begin{itemize}
\item $\rank_pX_\Sigma(K)_{\Gal(K/k_\infty)}\leq1$. 

\item $\rank_pX_\Sigma(K)_{\Gal(K/k'_\infty)}\leq1$. 
\end{itemize}
If the first inequality $\rank_pX_\Sigma(K)_{\Gal(K/k_\infty)}\leq1$ holds, we obtain $\rank_pX_\Sigma(k_\infty)\leq 2$ by Lemma~\ref{prankineq}. On the other hand, if the second inequality $\rank_pX_\Sigma(K)_{\Gal(K/k'_\infty)}\leq1$ holds, we obtain $\rank_pX_\Sigma(k'_\infty)\leq 2$ again by Lemma~\ref{prankineq}, and thus $\rank_pX_{\overline{\Sigma}}(k_\infty)\leq 2$ since $X_\Sigma(k'_\infty)$ and $X_{\overline{\Sigma}}(k_\infty)$ are isomorphic as $\Z_p$-modules. This completes the proof of Theorem~\ref{minimum rank 2}. 
\end{proof}
As a corollary to Theorem~\ref{minimum rank 2}, we obtain a theorem that determines all Iwasawa invariants of certain $\Z_p$-extensions. To see this, we refer the following theorem of Sands \cite{Sands1991}. 
\begin{thm}[{\cite[Theorem~3.1]{Sands1991}}]\label{SandsThm3.1}
Let $k/\Q$ be a finite extension, $p$ a prime number and $k_\infty/k$ a $\Z_p$-extension that is totally ramified at all ramified primes. Suppose that $\mu(k_\infty/k)=0$ and $X(k_\infty)$ has no non-trivial finite $\Z_p\bbr{\Gal(k_\infty/k)}$-submodule. Suppose also that $\lambda(k_\infty/k)<p-1$ or that $\lambda(k_\infty/k)=p-1$ and $p^2$ divides the constant term of the characteristic power series of $X(k_\infty)$. Then $v_p(h_{k_n})=\lambda(k_\infty/k)n+v_p(h_k)$ holds for all $n\geq 0$. 
\end{thm}
Now we are ready to prove Theorem~\ref{infinite Zp2}. 
\begin{proof}[Proof of Theorem~\ref{infinite Zp2}]
We have $\rank_pX(k_\infty)\leq 2$ by Theorem~\ref{minimum rank 2}. On the other hand, we have $\rank_{\Z_p}X(k_\infty)\geq 2$ by Lemma~\ref{ur}. These lead to the first assertion that $X(k_\infty)\simeq\Z_p^2$ as a $\Z_p\bbr{\Gal(k_\infty/k)}$-module. 

Now we consider the $\Z_p$-extension $k_\infty/k_e$, instead of $k_\infty/k$. Then $X(k_\infty)\simeq\Z_p^2$ as a $\Z_p\bbr{\Gal(k_\infty/k_e)}$-module and, in particular, $X(k_\infty)$ has no non-trivial finite $\Z_p\bbr{\Gal(k_\infty/k_e)}$-submodule. On the other hand, by Lemma~\ref{ur}, we have a surjection $X(k_\infty)\to\Gal(\widetilde{k}/k_\infty)$ of $\Z_p\bbr{\Gal(k_\infty/k_e)}$-modules. Hence the constant term of the characteristic power series is $0$ and, in particular, is divisible by $p^2$. Therefore, we can apply Theorem~\ref{SandsThm3.1} to obtain the assertion. 
\end{proof}
\begin{rem}
The vanishing of $\mu$-invariants can alternatively be seen by using Bloom--Gerth Theorem \cite[Theorem~1]{BG1981}. See Proposition~\ref{vanishing of mu} for details. 
\end{rem}
Here we mention that Theorem~\ref{SandsThm3.1} may also be used to obtain the following proposition, which refines \cite[Theorem~A~(2)]{Fujii2013}. 
\begin{prop}\label{FujiiThmA refine}
Let $k$ be an imaginary quadratic field and $p$ a prime number that splits in $k$. Suppose that $\lambda(k^\cy_\infty/k)=1$. Let $k_\infty/k$ be a $\Z_p$-extension that is ramified at each $p$-adic prime and $e\geq 0$ the minimum integer such that $k_\infty/k_e$ is totally ramified at each $p$-adic prime. Then $v_p(h_{k_n})=n+v_p(h_{k_e})-e$ holds for all $n\geq e$ and, in particular, $\mu(k_\infty/k)=0$, $\lambda(k_\infty/k)=1$ and $\nu(k_\infty/k)=v_p(h_{k_e})-e$. 
\end{prop}
\begin{proof}
We first claim that $X(k^\cy_\infty)\simeq\Z_p$. If $p$ is odd, it is well-known. For the case $p=2$,  see \cite[Theorem~7]{Ferrero1980}. Then, by Lemma~\ref{fiveterm} and Nakayama's Lemma, we obtain $X(\widetilde{k})=0$. By Lemmas \ref{fiveterm} and \ref{ur}, we obtain $X(k_\infty)\simeq\Z_p$ as a $\Z_p\bbr{\Gal(k_\infty/k)}$-module. The remaining part of the proof is similar to that of Theorem~\ref{infinite Zp2}. 
\end{proof}
\begin{rem}
In \cite{Sands1991}, Sands applied Theorem~\ref{SandsThm3.1} only to the cyclotomic $\Z_p$-extension over an imaginary quadratic field. In general, the assumption that $X(k_\infty)$ has no non-trivial finite $\Z_p\bbr{\Gal(k_\infty/k)}$-submodule is hard to be confirmed other than the case where $k_\infty/k$ is the cyclotomic $\Z_p$-extension. 
\end{rem}
\subsection{Some discussions on the vanishing of $\mu$}\label{Subsection Some discussions on the vanishing of mu}
In Theorem~\ref{infinite Zp2}, we showed the vanishing of $\mu$-invariants of certain $\Z_p$-extensions. However, we can show the vanishing of $\mu$-invariants by using Bloom--Gerth Theorem \cite[Theorem~1]{BG1981}. 
\begin{prop}\label{vanishing of mu}
Let $k$ be a quartic CM-field and $p$ an odd prime number. Suppose one of the following two conditions. 
\begin{enumerate}
\item $p$ splits completely in $k$ and $X(k^\cy_\infty)\simeq\Z_p^2$ as a $\Z_p\bbr{\Gal(k^\cy_\infty/k)}$-module. 

\item $p$ does not split in $k$ and $X(k^\cy_\infty)$ is cyclic as a $\Z_p$-module. 
\end{enumerate}
Then $\mu(k_\infty/k)=0$ for each $\Z_p$-extension $k_\infty/k$ such that $k_\infty\nsubseteq\widetilde{k}^\an$
\end{prop}
\begin{proof}
In both cases, we know $\mu(k^\cy_\infty/k)=0$ by the assumption. Let $k_\infty/k$ be a $\Z_p$-extension such that $\mu(k_\infty/k)>0$. By Bloom--Gerth theorem \cite[Theorem~1]{BG1981}, there is at most one $\Z_p$-extension over $k$ contained in $k^\cy_\infty k_\infty$ with positive $\mu$-invariant. Considering the action of the complex conjugation, the extension $k_\infty/k^+$ must be a Galois extension and we conclude that $k_\infty\subseteq\widetilde{k}^\an$. 
\end{proof}
We can immediately obtain a result on the vanishing of $m_0$-invariants of $\Z_p^2$-extensions over a quartic CM-field. Recall that the $m_0$-invariant $m_0(K_\infty/k)$ of a $\Z_p^2$-extension $K_\infty/k$ of a finite extension $k/\Q$ is defined as the largest integer such that $p^{m_0(K_\infty/k)}$ divides a generator of the characteristic ideal of $X(K_\infty)$ as a $\Z_p\bbr{\Gal(K_\infty/k)}$-module (see \cite[\S4]{Cuoco1980} for example). 
\begin{cor}
Suppose the conditions as in Proposition~\ref{vanishing of mu}. Then $m_0(K_\infty/k)=0$ for each $\Z_p^2$-extension $K_\infty/k$ such that $K_\infty\neq\widetilde{k}^\an$. 
\end{cor}
\begin{proof}
Let $K_\infty/k$ be a $\Z_p^2$-extension such that $m_0(K_\infty/k)>0$. Then, by \cite[Corollary~4.10]{Cuoco1980}, there are infinitely many $\Z_p$-extensions over $k$ contained in $K_\infty$ with positive $\mu$-invariants. 
By Proposition~\ref{vanishing of mu}, all of these $\Z_p$-extensions are contained in $\widetilde{k}^\an$. Thus we obtain $K_\infty=\widetilde{k}^\an$. 
\end{proof}
\subsection{Examples}\label{Subsection Examples}
\subsubsection{Strategy for constructing examples}
In this section, we show some examples that satisfy the assumptions of the theorems in \S\ref{Subsection Main theorems}. 

We start from the following observation. Let $k$ be a quartic CM-field and $p$ a prime number that splits completely in $k$. Moreover, suppose that $X(\widetilde{k})=0$. Then, for each $\Z_p$-extension $k_\infty/k$ that is ramified at each $p$-adic prime, we have $X(k_\infty)\simeq\Z_p^2$ by Lemma~\ref{ur}. So, under the condition that $X(\widetilde{k})=0$, we can deduce the assertions of Theorem~\ref{infinite Zp2} under weaker assumptions and in a quite simpler way. For this reason, we are mainly interested in examples such that $X(\widetilde{k})\neq 0$. We use the following criterion of Okano's, which is a part of \cite[Theorem~1.1.1~(II)]{Okano2025}. 
\begin{lem}[{\cite[Theorem~1.1.1~(II)]{Okano2025}}]\label{Okano criterion}
Let $k$ be a quartic CM-field and $p$ an odd prime number that splits  completely in $k$. Suppose that $\rank_p X(k)\geq 2$. Then $X(\widetilde{k})\neq 0$. 
\end{lem}
Using this criterion together with Fukuda's theorem \cite[Theorem~1]{Fukuda1994}, we obtain the following lemma. Here we recall that $F_1$ denotes the first layer of a given $\Z_p$-extension $F_\infty/F$, that is, the intermediate field of $F_\infty/F$ such that $[F_1:F]=p$. 
\begin{lem}\label{condition}
Let $k$ be a quartic CM-field and $p$ an odd prime number that splits completely in $k$. Suppose that $\rank_p X(k)=\rank_p X(k^\cy_1)=2$ holds. Then $X(k^\cy_\infty)\simeq \Z_p^2$ as a $\Z_p\bbr{\Gal(k^\cy_\infty/k)}$-module and $X(\widetilde{k})\neq 0$. 
\end{lem}
\begin{proof}
We first remark that $k^\cy_\infty/k$ is totally ramified at each $p$-adic prime because $p$ splits completely in $k$. Then we have $\rank_pX(k^\cy_\infty)=2$ by the assumption and \cite[Theorem~1~(2)]{Fukuda1994}. Since we have $\rank_{\Z_p}X(k^\cy_\infty)\geq 2$ by Lemma~\ref{ur}, we see that $X(k^\cy_\infty)\simeq\Z_p^2$ as a $\Z_p\bbr{\Gal(k^\cy_\infty/k)}$-module (see Remark~\ref{trivial action}). The assertion $X(\widetilde{k})\neq 0$ follows directly from Lemma~\ref{Okano criterion}. 
\end{proof}
The following class group computation was performed using PARI/GP \cite{PARI2}, under the assumption of the GRH (Generalized Riemann Hypothesis). 
\subsubsection{Biquadratic fields}
Let $p=3$. For $k^+=\Q(\sqrt{m})$ with $m$ listed in the following table, the prime $3$ splits in $k^+$ and we can check $X((k^+)^\cy_1)=0$. Put $k\coloneqq\Q(\sqrt{m}, \sqrt{-d})$ for a square-free positive integer $d$. If $\Q(\sqrt{m}, \sqrt{-d})=\Q(\sqrt{m}, \sqrt{-d'})$ for another square-free positive integer $d'$, then we chose the smaller one as $d$. For each $m$, we check the conditions in Lemma~\ref{condition} for $k$ in the range $1\leq d\leq 10000$. Table~\ref{biquad} contains the number of $k$ that satisfy the conditions in this range and the first $10$ of them, arranged by the size of $d$. 
\begin{table}[H]
\centering
\begin{tabular}{|c||c|c|}
\hline
$m$ & number of $k$ & first $10$ of those $d$ \\ \hline \hline
$7$ & $75$  & $26, 431, 473, 563, 566, 755, 821, 1055, 1361, 1397$. \\ \hline
$10$ & $67$  & $89, 557, 707, 782, 839, 914, 959, 1118, 1142, 1322$. \\ \hline
$13$ & $82$  & $329, 491, 527, 794, 905, 989, 1142, 1166, 1397, 1439$. \\ \hline
$19$ & $86$ & $110, 170, 329, 491, 515, 593, 839, 983, 1055, 1142$. \\ \hline
$22$ & $75$ & $53, 329, 335, 431, 434, 731, 1106, 1313, 1502, 1517$. \\ \hline
$31$ & $91$ & $233, 542, 671, 677, 707, 794, 821, 839, 959, 1067$. \\ \hline
$34$ & $83$ & $23, 59, 89, 335, 431, 473, 491, 557, 707, 794$. \\ \hline
$37$ & $77$ & $29, 170, 182, 335, 497, 665, 1145, 1166, 1169, 1394$. \\ \hline
$46$ & $79$ & $83, 89, 170, 431, 497, 563, 593, 695, 755, 905$. \\ \hline
\end{tabular}
\caption{~Examples of biquadratic fields}
\label{biquad}
\end{table}
\subsubsection{Cyclic fields}
Let $f(s, t; x)\coloneqq x^4+2s(t^2+1)x^2+s^2t^2(t^2+1)\in\Q(s, t)[x]$. We use this polynomial to construct imaginary cyclic quartic fields. This polynomial is known as a generic polynomial of $C_4$ (the cyclic group of order $4$). See \cite[Corollary~2.2.6]{JLY2002} for details. For each pair $(s, t)$  in Table~\ref{cyclic}, the field $k$ defined by $f(s, t; x)$ satisfy the conditions in Lemma~\ref{condition}. 
\begin{table}[H]
\centering
\begin{tabular}{|c|c||c|}
\hline
$t$ & $k^+$ & $s$ \\ \hline \hline
$3$ & $\Q(\sqrt{10})$ & $43, 103, 166, 214, 367, 397, 403, 415, 535, 553$. \\ \hline
$3/2$ & $\Q(\sqrt{13})$ & $109, 115, 145, 331, 355, 373, 454, 493, 526, 589$. \\ \hline
$5/3$ & $\Q(\sqrt{34})$ & $65, 107, 110, 137, 227, 314, 317, 359, 419, 467$. \\ \hline
$6$ & $\Q(\sqrt{37})$ & $31, 43, 46, 58, 118, 157, 163, 262, 391, 502$. \\ \hline
$6/5$ & $\Q(\sqrt{61})$ & $223, 253, 307, 355, 367, 463, 493, 589, 655, 730$. \\ \hline
\end{tabular}
\caption{~Examples of cyclic fields}
\label{cyclic}
\end{table}
\subsubsection{Non-Galois fields}
For each non-negative integers $m$ and $d$, we put $k\coloneqq\Q(\sqrt{\sqrt{m}-d})$. For each $m$ listed in Table~\ref{non-Galois}, we check the conditions in Lemma~\ref{condition} for $k$ in the range $1\leq d\leq 20000$, and the table contains all of them that satisfy conditions in~Lemma~\ref{condition}. 
\begin{table}[H]
\centering
\begin{tabular}{|c||c|c|}
\hline
$m$ & $d$ \\ \hline \hline
$7$& $8347, 17338$. \\ \hline
$10$ & $4744, 7381, 8542, 8866, 14995$. \\ \hline
$13$ & $250, 11806, 11914, 13543$. \\ \hline
$19$ & $1027, 1864, 1945, 9001, 11908, 18874$. \\ \hline
$22$ & $7882, 7963, 19411$. \\ \hline
$31$ & $2824, 5740, 11194, 15433$. \\ \hline
$34$ & $760, 3244, 6889, 11290, 13666, 16339$. \\ \hline
$37$ & $685, 5221, 8488, 9460, 13834$. \\ \hline
$46$ & $5887, 8749, 9586, 9883, 17470$. \\ \hline
\end{tabular}
\caption{~Examples of non-Galois fields}
\label{non-Galois}
\end{table}
\section{Iwasawa modules over a CM-field}\label{Section Iwasawa modules over a CM-field}
The following is a generalization of Theorem~\ref{FujiiThm1} and partially generalizes Theorem~\ref{minimum rank 2}. 
\begin{thm}\label{minimum rank d1}
Let $k$ be a CM-field of degree $2d$ with a positive integer $d$. Let $p$ be an odd prime number that splits completely in $k$ and $\Sigma$ a set of $p$-adic primes of $k$ that satisfies $\Sigma\cap\overline{\Sigma}=\emptyset$. Assume that Leopoldt's Conjecture holds for $k$ and $p$. Suppose that $X_\Sigma(\widetilde{k})$ is cyclic as a $\Z_p\bbr{\Gal(\widetilde{k}/k^\cy_\infty)}$-module. Then we have
\[
\min\{\rank_pX_\Sigma(k_\infty), \rank_pX_{\overline{\Sigma}}(k_\infty)\}\leq d+1
\]
for each $\Z_p$-extension $k_\infty/k$ that satisfies $k_\infty\cap\widetilde{k}^\an=k$. 
\end{thm}
\begin{proof}
By Lemma~\ref{anti-cyclotomic intersection}, we can take an anti-cyclotomic $\Z_p$-extension $k^\an_\infty/k$ such that $k_\infty\subseteq k^\cy_\infty k^\an_\infty$,  and this satisfies $k_\infty\cap k_\infty^\an=k$. Put $K\coloneqq k^\cy_\infty k^\an_\infty$. By the assumption that $X_\Sigma(\widetilde{k})$ is cyclic as a $\Z_p\bbr{\Gal(\widetilde{k}/k^\cy_\infty)}$-module, $X_\Sigma(\widetilde{k})_{\Gal(\widetilde{k}/K)}$ is cyclic as a $\Z_p\bbr{\Gal(K/k^\cy_\infty)}$-module. Thus, by Theorem~\ref{First Theorem} applied to $X=X_\Sigma(\widetilde{k})_{\Gal(\widetilde{k}/K)}$, we have one of the following two inequalities.
\begin{itemize}
\item $\rank_pX_\Sigma(\widetilde{k})_{\Gal(\widetilde{k}/k_\infty)}\leq1$. 

\item $\rank_pX_\Sigma(\widetilde{k})_{\Gal(\widetilde{k}/k'_\infty)}\leq1$. 
\end{itemize}
If the first inequality $\rank_pX_\Sigma(\widetilde{k})_{\Gal(\widetilde{k}/k_\infty)}\leq1$ holds, we obtain $\rank_pX_\Sigma(k_\infty)\leq d+1$ by Lemma~\ref{prankineq}. On the other hand, if the second inequality $\rank_pX_\Sigma(\widetilde{k})_{\Gal(\widetilde{k}/k'_\infty)}\leq1$ holds, we obtain $\rank_pX_\Sigma(k'_\infty)\leq d+1$ again by Lemma~\ref{prankineq}, and thus $\rank_pX_{\overline{\Sigma}}(k_\infty)\leq d+1$ since $X_\Sigma(k'_\infty)$ and $X_{\overline{\Sigma}}(k_\infty)$ are isomorphic as $\Z_p$-modules. 
\end{proof}
\begin{cor}\label{minimum rank pe 2}
Let $k$ be an imaginary quadratic field, $p$ an odd prime number that splits in $k$ and $\pe$ a $p$-adic prime of $k$. Suppose that $\lambda(k^\cy_\infty/k)=2$. Then we have
\[
\min\{\rank_pX_\pe(k_\infty), \rank_pX_{\overline{\pe}}(k_\infty)\}\leq 2
\]
for each $\Z_p$-extension $k_\infty/k$ that satisfies $k_\infty\cap k^\an_\infty=k$. 
\end{cor}
\begin{proof}
We have an exact sequence
\[
\begin{tikzpicture}
\node (a) at (-0.5,0) {$0$}; 
\node (b) at (2,0) {$X_\pe(\widetilde{k})_{\Gal(\widetilde{k}/k^\cy_\infty)}$}; 
\node (c) at (5,0) {$X(k^\cy_\infty)$}; 
\node (d) at (8,0) {$\Gal(\widetilde{k}/k^\cy_\infty)$}; 
\node (e) at (10.5,0) {$0$};
\draw[->] (a) -- (b); \draw[->] (b) -- (c); \draw[->] (c) -- (d); \draw[->] (d) -- (e); 
\end{tikzpicture}
\]
by Lemmas \ref{fiveterm} and \ref{M_Sigma L}. Since $X(k^\cy_\infty)\simeq\Z_p^2$, we have $X_\pe(\widetilde{k})_{\Gal(\widetilde{k}/k^\cy_\infty)}\simeq\Z_p$, and thus $X_\pe(\widetilde{k})$ is cyclic as a $\Z_p\bbr{\Gal(\widetilde{k}/k^\cy_\infty)}$-module by Nakayama's Lemma. Hence we may apply Theorem~\ref{minimum rank d1} to obtain the assertion.  
\end{proof}
\begin{thm}\label{infinite Zpd}
Let $k$ be a CM-field of degree $2d$ with a positive integer $d$. Let $p$ be an odd prime number that splits completely in $k$. Assume that Leopoldt's Conjecture holds for $k$ and $p$. Suppose that $X(\widetilde{k})$ is cyclic as a $\Z_p\bbr{\Gal(\widetilde{k}/k^\cy_\infty)}$-module. Let $k_\infty/k$ be a $\Z_p$-extension that satisfies the following.  
\begin{itemize}
\item $k_\infty/k$ is ramified at each $p$-adic prime. 

\item $k_\infty\cap\widetilde{k}^\an=k$. 
\end{itemize}
Then $X(k_\infty)\simeq\Z_p^d\oplus\left(\Z_p/p^e\Z_p\right)$ as a $\Z_p$-module for some $e\in\Z_{\geq 0}\cup\{\infty\}$, where we use the convention $\Z_p/p^\infty\Z_p\coloneqq\Z_p$. In particular, $\mu(k_\infty/k)=0$ and $d\leq\lambda(k_\infty/k)\leq d+1$. 
\end{thm}
\begin{proof}
We have $\rank_pX(k_\infty)\leq d+1$ by Theorem~\ref{minimum rank d1}. On the other hand, we have $\rank_{\Z_p}X(k_\infty)\geq d$ by Lemma~\ref{ur}. These lead to the assertion. 
\end{proof}
\begin{rem}
Under the assumptions in Theorem~\ref{minimum rank 2}, the assumptions in Theorem~\ref{minimum rank d1} hold. Indeed, by Lemmas \ref{ur} and \ref{fiveterm}, $X_\Sigma(\widetilde{k})_{\Gal(\widetilde{k}/k^\cy_\infty)}$ is cyclic as a $\Z_p$-module, and thus $X_\Sigma(\widetilde{k})$ is cyclic as a $\Z_p\bbr{\Gal(\widetilde{k}/k^\cy_\infty)}$-module by Nakayama's Lemma. However, Theorem~\ref{minimum rank d1} does not imply Theorem~\ref{minimum rank 2}. We also remark that we have no numerical applications of Theorem~\ref{infinite Zpd} yet, except for the case where $k$ is an imaginary quadratic field, in which case we obtained Corollary~\ref{minimum rank pe 2}. 
\end{rem}
\section*{Acknowledgments}
I would like to express my sincere gratitude to my supervisor Takenori Kataoka for the numerous valuable suggestions. I am also grateful to Yuta Katayama for providing helpful knowledge on constructing examples. I am thankful to the referee for the valuable comments. In particular, Lemma~\ref{condition} was improved thanks to the referee's observation.

\printbibliography
\end{document}